\documentclass[10pt]{article}

\usepackage{amsmath}
\usepackage{amsthm} 
\usepackage{amsfonts}
\usepackage{latexsym}
\usepackage{amssymb}  
\usepackage{epsfig}
\usepackage{times}
\usepackage{color}

\usepackage{epsfig}

\usepackage{url}

\newcommand{\R}{\mathbb{R}}
\newcommand{\N}{\mathbb{N}}

\DeclareMathOperator{\dom}{dom}
\DeclareMathOperator{\Dom}{Dom}

\DeclareMathOperator{\range}{Ran}
\DeclareMathOperator{\graph}{Gr}
\DeclareMathOperator{\proy}{Proj}
\DeclareMathOperator{\epi}{epi}

\newcommand{\eps}{\varepsilon}

\newcommand{\inner}[2]{\langle{#1},{#2}\rangle}

\newcommand{\ds}{\displaystyle}

\newcommand{\tos}{\rightrightarrows} 

\newtheorem{theorem}{Theorem}
\newtheorem{lemma}[theorem]{Lemma}

\newtheorem{proposition}[theorem]{Proposition}

\newtheorem{fact}{Fact}

\theoremstyle{definition}
 
\newtheorem{definition}[theorem]{Definition}

\title{Lower Limits of Type (D) Monotone Operators in general Banach Spaces}
\author{Orestes Bueno\thanks{IMCA - Instituto de Matem\'aticas y Ciencias Afines. Email: \texttt{obueno@imca.edu.pe}} \and Yboon Garc\'ia\thanks{CIUP - Centro de Investigaci\'on de la Universidad del Pac\'ifico. Email: \texttt{garcia\_yv@up.edu.pe}} \and Maicon Marques Alves\thanks{Universidade Federal de Santa Catarina. Email: \texttt{maicon.alves@ufsc.br}}}

\begin{document}

\maketitle

\begin{abstract}
We give, for general Banach spaces, a characterization of the sequential lower limit of maximal monotone operators of type (D) and prove its representability. As a consequence, using a recent extension of the Moreau-Yosida regularization for type (D) operators, we extend to general Banach spaces the definitions of the variational sum of monotone operators and the variational composition of monotone operators with continuous linear mappings, and we prove that both operators are representable.
\bigskip


\end{abstract}

\section{Introduction}
The sequential lower limit of maximal monotone operators was studied by Garc\'ia and Lassonde~\cite{GarciaLassonde2012}, for reflexive Banach spaces provided with a strictly convex norm and dual norm, and then applied to prove the representability of the \emph{variational sum} (see \cite{AtBaTh94,RevThe99-1,RevThe99-2}) and the \emph{variational composition} (see \cite{PeReTh03}). The key tools for such work were Minty-Rockafellar surjectivity theorem (see Fact~\ref{fact:rock70} or~\cite{Rock70}) and the Moreau-Yosida regularization, adapted for reflexive Banach spaces by Brezis, Crandall and Pazy~\cite{BCP70}.

Unfortunately, such tools are not available in general for maximal monotone operators in non-reflexive Banach spaces. 
However, thanks to some recent work of Marques Alves and Svaiter~\cite{MAS11,BM2011}, the aforementioned results are available for maximal monotone operators \emph{of type (D)}, introduced by Gossez~\cite{JPGos0}. Observe that type (D) operators have many nice properties that maximal monotone operators have in reflexive Banach spaces.
Thus, for this kind of operators, we have a weak version of Minty-Rockafellar theorem (using $J_{\eps}$ instead of the duality mapping $J$, see Fact~\ref{fact:typedeq}) and a general version of the Moreau-Yosida regularization (Fact~\ref{fact:myreg}).

In this paper, we extend the characterization of the lower limit given in~\cite{GarciaLassonde2012} to any Banach space (Proposition~\ref{prop:main}), and prove that the lower limit of a sequence of type (D) operators is representable (Theorem~\ref{teo:main}).  We also extend the definition of variational sum and variational composition (section~\ref{sec:varsumcomp}) to any Banach space when the operators involved are of type (D) and prove that the variational sum and the variational composition are representable (Theorems~\ref{teo:varcomp} and \ref{teo:varsum}). 

We must emphasize that our study is not a simple generalization from reflexive to non-reflexive, since even in the reflexive case we can drop the norm-regularity conditions asked in~\cite{GarciaLassonde2012}.

\section{Preliminary definitions and notations}
Let $U,V$ be non-empty sets. For a set $S\subset U\times V$, we denote the projections over $U$ and $V$ as $\proy_U(S)=\{u\in U\::\:\exists\,v\in V,\,(u,v)\in S\}$ and 
$\proy_V(S)=\{v\in V\::\:\exists\,u\in U,\,(u,v)\in S\}$.

A \emph{multivalued operator} $T:U\tos V$ is an application $T:U\to \mathcal{P}(V)$, that is, for $u\in U$, $T(u)\subset V$. The \emph{graph}, \emph{domain} and \emph{range} of $T$ are defined, respectively, as
\begin{gather*}
\graph(T)=\big\{(u,v)\in U\times V\::\: v\in T(u)\big\},\\
\Dom(T)=\big\{u\in U\::\: T(u)\neq\emptyset\big\}=\proy_U(\graph(T)),\\
\range(T)=\bigcup_{u\in U}T(u)=\proy_V(\graph(T)).
\end{gather*}
In addition, $T^{-1}:V\tos U$ is defined as $T^{-1}(v)=\{u\in U\::\:v\in T(u)\}$. If $V$ is a vector space, for multivalued operators $T,S:U\tos V$, their \emph{sum} $T+S:U\tos V$ is the operator defined as $(T+S)(u)=T(u)+S(u)=\{v+w\::\:v\in T(u),\,w\in S(u)\}$.
From now on, we will identify multivalued operators with their graphs, so we will write $(u,v)\in T$ instead of $(u,v)\in \graph(T)$.

Let $X$ be a Banach space and $X^*$, $X^{**}$ be its topological dual and bidual respectively. We will identify $X$ with its image under the canonical injection into $X^{**}$, so in this way $X$ is called \emph{reflexive} if $X=X^{**}$ and \emph{non-reflexive}, otherwise. The \emph{duality product} $\pi:X\times X^*\to\R$ is defined as $\pi(x,x^*)=\inner{x}{x^*}=x^*(x)$. 

We will say that $(x,x^*),(y,y^*)\in X\times X^*$ are \emph{monotonically related} if 
\[
\inner{x-y}{x^*-y^*}\geq 0,
\]
and this will be denoted by $(x,x^*)\sim(y,y^*)$. Analogously, $(x,x^*)$ is \emph{monotonically related} to an operator $T$ if it is monotonically related to every point in $T$, and this will be denoted by $(x,x^*)\sim T$.
A multivalued operator $T:X\tos X^*$ is called \emph{monotone} if every $(x,x^*),(y,y^*)\in T$ are monotonically related, and it is called \emph{maximal monotone} if it is monotone and $T\subset S$ and $S$ monotone implies $T=S$. In the same way, an operator $T:X^*\to X$ is (maximal) monotone if $T^{-1}$ is so too.
We can also consider the \emph{monotone polar} of $T$ defined as
\[
T^{\mu}=\{(x,x^*)\in X\times X^*\::\:(x,x^*)\sim T\}.
\]

Given a function $f:X\to\R\cup\{+\infty\}$, we denote its \emph{effective domain} by $\dom(f)=\{x\in X\::\: f(x)<\infty\}$ and its \emph{epigraph} as $\epi(f)=\{(x,\lambda)\in X\times \R\::\:f(x)\leq \lambda\}$. We respectively call $f$ \emph{proper}, \emph{convex} and \emph{lower semicontinuous}, whenever $\epi(f)$ is a non-empty, convex and closed set in $X\times \R$.
Let $f:X\to\R\cup\{+\infty\}$ be a lower semicontinuous, proper, convex function. The \emph{Fenchel conjugate} $f^*:X^*\to\R\cup\{+\infty\}$ is defined as $f^*(x^*)=\ds\sup_{x\in X}\{\inner{x}{x^*}-f(x)\}$ and is also lower semicontinuous, proper and convex, whenever $f$ is. The \emph{subdifferential} of $f$ is the multivalued operator defined as 
\[
\partial f=\big\{(x,x^*)\::\:f(y)\geq f(x)+\inner{y-x}{x^*},\,\forall\,y\in X\big\}.
\]
In the same way, the \emph{$\eps$-subdifferential} of $f$ is defined as
\[
\partial_{\eps} f=\big\{(x,x^*)\::\:f(y)\geq f(x)+\inner{y-x}{x^*}-\eps,\,\forall\,y\in X\big\}. 
\]
The subdifferential can be related to the Fenchel conjugate, recall the Fenchel-Young inequality,
\[
f(x)+f^*(x^*)\geq \inner{x}{x^*},
\]
where the equality holds if, and only if, $(x,x^*)\in \partial f$. In the same way, $(x,x^*)\in\partial_{\eps}f$ if, and only if, 
\[
\inner{x}{x^*}\leq f(x)+f^*(x^*)\leq\inner{x}{x^*}+\eps.
\]
For instance, if $j:X\to\R$, $j(x)=\frac{1}{2}\|x\|^2$, we have $j^*:X^*\to\R$, $j^*(x)=\frac{1}{2}\|x^*\|^2$ and we will denote $J=\partial\frac{1}{2}\|\cdot\|^2$ and $J_{\eps}=\partial_{\eps}\frac{1}{2}\|\cdot\|^2$. 

Let $T:X\tos X^*$ be a maximal monotone operator, denote $\widehat T,\widetilde T:X^{**}\tos X^*$ as 
\[
\widehat T=T,\qquad\text{and}\qquad \widetilde{T}=\{(x^{**},x^*)\in X^{**}\times X^*\::\:(x^{**},x^*)\sim \widehat T\}.
\]
The following definition was due to Gossez.
\begin{definition}[{\cite{JPGos0}}]\label{def:typeD}
We say that $T$ is of \emph{type (D)}, if for every $(x^{**},x^*)\in \widetilde{T}$, there exists a bounded net $(x_{\alpha},x^*_{\alpha})_{\alpha}\subset T$ such that $(x_{\alpha},x_{\alpha}^*)\to (x^{**},x^*)$ in the weak$^*\times $strong topology of $X^{**}\times X^*$.
 \end{definition}
For instance, subdifferentials of convex, proper and lower semicontinuous functions are of type (D).

Another equivalent class of maximal monotone operators was given by Simons.
\begin{definition}[{\cite{Simons96}}]
We say that $T$ is of \emph{type (NI)} if for all $(x^{**},x^*)\in X^{**}\times X^*$,
\[
\inf_{(y,y^*)\in T}\inner{x^{**}-y}{x^*-y^*}\leq 0.
\]
\end{definition}
Simons proved that type (D) implies type (NI) in \cite{Simons96}, and Marques Alves and Svaiter proved the full equivalency in~\cite{BSMMA-TypeD}.

When $X$ is non-reflexive, if $T$ is a type (D) operator, then $\widehat T$ has a unique maximal monotone extension on $X^{**}\times X^*$ and, in this case, such maximal extension is $\widetilde{T}$. Note that every maximal monotone operator in a reflexive space is trivially of type (D).

A convex, proper and lower semicontinuous function $h:X\times X^*\to\R\cup\{+\infty\}$ is a \emph{representative} of a monotone operator $T:X\tos X^*$ if $h(x,x^*)\geq \inner{x}{x^*}$, for all $(x,x^*)\in X\times X^*$, and $T=\{(x,x^*)\in X\times X^*\::\:h(x,x^*)=\inner{x}{x^*}\}$. If $T$ has a representative, then it is called a \emph{representable} operator. 
The theory of convex representations began with the seminal work of Fitzpatrick~\cite{Fitz}, which was independently rediscovered by Burachik and Svaiter~\cite{BS}, and Martinez-Legaz and Th\'era~\cite{MLThe}. Fitzpatrick defined, for a monotone operator $T:X\tos X^*$, its \emph{Fitzpatrick function} $\varphi_T:X\times X^*\to\R$,
\[
\varphi_T(x,x^*)=\inner{x}{x^*}-\inf_{(y,y^*)\in T}\inner{y-x}{y^*-x^*}.
\]
\begin{fact}[{\cite[Theorem 3.4]{Fitz}}]\label{fact:fitz}
Let $T:X\tos X^*$ be a monotone operator and $(x,x^*)\in T$. Then $\varphi_T(x,x^*)=\inner{x}{x^*}$,  $(x^*,x)\in\partial\varphi_T(x,x^*)$, and  $\varphi_T^*(x^*,x)=\inner{x}{x^*}$.
\end{fact}
Fitzpatrick also proved that maximal monotone operators are representable by $\varphi_T$. The reader can find more information about convex representations in~\cite{Penot04,BSML,Voisei06}, and the references therein.
\smallskip 

A representative function $h:X\times X^*\to\R\cup\{+\infty\}$ of a monotone operator $T$ is called a \emph{strong representative} of $T$ if $h^*(x^*,x^{**})\geq \inner{x^*}{x^{**}}$, for all $(x^*,x^{**})\in X^*\times X^{**}$.  The following fact was proven by Marques Alves and Svaiter.
\begin{fact}[{\cite[Theorem 3.6]{BM2011} and \cite[Theorem 4.4]{BSMMA-TypeD}}]\label{fact:typedeq}
For a maximal monotone operator $T:X\tos X^*$, the following conditions are equivalent:
\begin{enumerate}
\item $T$ is of type (D);
\item $T$ has a strong representative;
\item $\range(T+J_{\eps}(\cdot-x_0))=X^*$, for all $x_0\in X$ and $\eps>0$. 
\end{enumerate}
\end{fact}
\smallskip

It is well known the following result due to Rockafellar.
\begin{fact}[{\cite{Rock70}}]\label{fact:rock70}
Let $X$ be reflexive and $T:X\tos X^*$ monotone. Then $T$ is maximal if, and only if, $\range(T+J)=X^*$. 
\end{fact}
As a consequence of this result, when $X$ is reflexive with a strictly convex norm and strictly convex dual, given a maximal monotone operator $T:X\tos X^*$ and $\lambda>0$, the inclusion
\[
0\in \lambda T(\cdot)+J(\cdot-x)
\]
has a solution $R_{\lambda}(x)$, which is also unique. Thus, the \emph{resolvent} of $T$ is the operator $R_{\lambda}:X\to X$ and the \emph{Moreau-Yosida regularization} of $T$ is the operator $T_{\lambda}:X\to X^*$, defined as $T_{\lambda}(x)=\dfrac{1}{\lambda}J(x-R_{\lambda}(x))$.
These operators were studied by Moreau and Yosida in Hilbert spaces, and generalized to reflexive spaces by Brezis \emph{et al.}~\cite{BCP70}. Recently, Marques Alves and Svaiter~\cite{MAS11} extended these definitions to general Banach spaces, for maximal monotone operators of type (D). For $X$ any Banach space and given $T:X\tos X^*$ a type (D) operator and $\lambda>0$, we consider the solutions $(x^*,z^{**})$ of the system
\begin{equation}\label{eq:myreg}
x^*\in \widetilde{T}(z^{**}),\qquad 0\in \lambda x^*+\widetilde J(z^{**}-x).
\end{equation}
Thus, the new versions of $R_{\lambda}$ and $T_{\lambda}$ are 
\begin{gather*}
R_{\lambda}:X\tos X^{**},\quad \graph(R_{\lambda})=\left\{(x,z^{**})\::\:
\begin{array}{l}
\exists\,x^*\in X^*\text{ such that}\\
(x^*,z^{**})\text{ is solution of~\eqref{eq:myreg}}
\end{array}\right\}\\
T_{\lambda}:X\tos X^{*},\quad \graph(T_{\lambda})=\left\{(x,x^*)\::\:
\begin{array}{l}
\exists\,z^{**}\in X^{**}\text{ such that}\\
(x^*,z^{**})\text{ is solution of~\eqref{eq:myreg}}
\end{array}\right\}
\end{gather*}
Some properties of $T_{\lambda}$ are the following.
\begin{fact}[{\cite[Theorem 4.6]{MAS11}}]\label{fact:myreg} Let $X$ be a Banach space, $\lambda>0$ and let $T:X\tos X$ be a maximal monotone operator of type (D). Then the following holds
\begin{enumerate}
 \item $T_{\lambda}$ is maximal monotone of type (D),
 \item $\Dom(T_{\lambda})=X$,
 \item $T_{\lambda}$ maps bounded sets into bounded sets	.
\end{enumerate}
\end{fact}

\section{Sequential lower limits in general Banach spaces}
Let $X$ be a Banach space and let $\{T_n:X\tos X^*\}_{n\in\N}$ be a sequence of maximal monotone operators.
The \emph{sequential lower limit} (with respect to the strong$\times $strong topology) of $\{T_n\}_n$ is
\[
T=\liminf T_n=\big\{(x,x^*)\::\:\exists\,(y_n,y_n^*)\in T_n,\text{ s.t. }\lim_n(y_n,y_n^*)=(x,x^*)\big\},
\]
where the convergence is taken in the strong$\times $strong topology of $X\times X^*$.
Garc\'ia and Lassonde~\cite{GarciaLassonde2012}, in reflexive Banach spaces, proved the following characterization of $T$: $(x,x^*)\in T$ if, and only if, $x=\lim x_n$, where $x_n$ is the unique solution of
\[
x^*\in T_n(x_n)+J(x_n-x).
\]
Note that Fact~\ref{fact:rock70} is central for $(x_n)_n$ to be well defined.

Our aim is to recover this result for a sequence of type (D) operators in any Banach space, so from %
now on, $X$ will be any real Banach space.

Let $\{T_n:X\tos X^*\}_{n\in\N}$ be a sequence of maximal monotone operators of type (D) and let $T:X\tos X^*$ be its lower limit.
For every $T_n$, consider $\widetilde{T}_n:X^{**}\tos X^*$ be its unique extension to the bidual, which exists as $T_n$ is of type (D). Thus, the sequence $\{\widetilde T_n\}_n$ also has its lower limit, say $S:X^{**}\tos X^{*}$. 
Note that $T_n\subset \widetilde{T}_n$, since $T_n$ is monotone, and this implies $T\subset S$, considering $X\subset X^{**}$.

Consider ${(\eps_n)}_n$ to be any sequence of positive numbers approaching to 0. Given $(x,x^*)\in X\times X^*$ and $n\in\N$, by Fact~\ref{fact:typedeq}, the inclusion
\begin{equation}\label{eq:eq1}
x^*\in T_n(x_n)+J_{\eps_n}(x_n-x),
\end{equation}
has a solution $x_n\in X$. Thus, there exist $x_n^*\in T_n(x_n)$ and $w_n^*\in J_{\eps_n}(x_n-x)$ such that $x^*=x_n^*+w_n^*$.

\begin{lemma}\label{lem:tech}
For any $(x,x^*)\in X\times X^*$, let $(x_n)_n$ be the sequence of solutions of the inclusion~\eqref{eq:eq1} and let ${(x_n^*)}_n,{(w_n^*)}_n\subset X^*$ such that $x^*=x_n^*+w_	n^*$, $x_n^*\in T_n(x_n)$ and $w_n^*\in J_{\eps_n}(x_n-x)$.
Then $(x_n)_n$, ${({x^*_n})}_n$ and ${({w^*_n})}_n$ are bounded. Moreover, given any weak$^*$-limit points $\bar x^{**}$ and $\bar w^*$ of $(x_n)_n$ and ${(w_n^*)}_n$, respectively, there exist subnets ${(x_{n_\alpha})}_{\alpha}$ and ${(w^*_{n_\alpha})}_{\alpha}$ of ${(x_n)}_n$ and ${(w^*_n)}_n$, respectively, such that
\begin{equation}
\inner{u^{**}-\bar x^{**}}{u^*-(x^*-\bar w^*)}\geq \dfrac{1}{2}\lim_{\alpha}\|x_{n_{\alpha}}-x\|^2+\dfrac{1}{2}\lim_{\alpha}\|w^*_{n_{\alpha}}\|^2+\inner{x-\bar x^{**}}{\bar w^*},\label{eq:forfitz}
\end{equation}
for all $(u^{**},u^*)\in S$. In particular, $(\bar x^{**},x^*-\bar w^*)\sim S$.
\end{lemma}
\begin{proof}
Let $(u^{**},u^*)\in S$ be fixed and let $(u_n^{**},u_n^*)\in \widetilde T_n$ such that $(u_n^{**},u_n^*)$ converges  strongly to $(u^{**},u^*)$. Using the monotonicity of $\widetilde T_n$ and the facts that $T_n\subset \widetilde T_n$ and $(x_n,x_n^*)\in T_n$, we have
\[
\inner{u_n^{**}-x_n}{u_n^{*}-x_n^*}\geq 0,
\]
and this, together with $x^*=x_n^*+w_n^*$, implies
\begin{equation}\label{eq:eq1.5}
\inner{u_n^{**}-x_n}{u_n^{*}-x^*}\geq\inner{x_n}{w_n^*}-\inner{u_n^{**}}{w_n^*}.
\end{equation}
Moreover, since $w_n^*\in J_{\eps_n}(x_n-x)$, 
\begin{equation}\label{eq:eq2}
\dfrac{1}{2}\|x_n-x\|^2+\dfrac{1}{2}\|w_n^*\|^2\leq 
\inner{x_n}{w_n^*}-\inner{x}{w_n^*}+\eps_n,
\end{equation}
so we replace~\eqref{eq:eq2} on~\eqref{eq:eq1.5} to obtain
\begin{align}
\inner{u_n^{**}-x_n}{u_n^{*}-x^*}&\geq\inner{x_n}{w_n^*}-\inner{u_n^{**}}{w_n^*}\notag\\
				 &\geq \dfrac{1}{2}\|x_n-x\|^2+\dfrac{1}{2}\|w_n^*\|^2+\inner{x-u_n^{**}}{w_n^*}-\eps_n.\label{eq:eq3}
\end{align}
On the other hand, $(\|u_n^{**}-x\|-\|w_n^*\|)^2\geq 0$ implies
\[
\inner{u_n^{**}-x}{w_n^*}
\leq\dfrac{1}{2}\|u_n^{**}-x\|^2+\dfrac{1}{2}\|w_n^*\|^2.
\]
This inequality together with~\eqref{eq:eq3} shows that
\begin{align}
\inner{u_n^{**}-x_n}{u_n^{*}-x^*}&\geq \dfrac{1}{2}\|x_n-x\|^2+\dfrac{1}{2}\|w_n^*\|^2+\inner{x-u_n^{**}}{w_n^*}-\eps_n\label{eq:eq3.5}\\
&\geq \dfrac{1}{2}\|x_n-x\|^2-\dfrac{1}{2}\|u_n^{**}-x\|^2-\eps_n.\label{eq:eq4}
\end{align}
The right side in equation~\eqref{eq:eq4} has a quadratic term depending on $(x_n)_n$, which is bounded by a linear one on the left side of~\eqref{eq:eq3.5}.
This, along with the fact that ${(u_n^{**})}_n$ and ${(u_n^{*})}_n$ are bounded, implies that $(x_n)_n$ is bounded. Similarly, rearranging~\eqref{eq:eq3.5}, we obtain that ${(w^*_n)}_n$ is also bounded, and so is ${(x_n^*)}_n$.

Take any weak$^*$-limit points of $(x_n)_n$ and ${(w^*_n)}_n$, say $\bar x^{**}$ and $\bar w^*$, which exist by the Banach-Alaoglu Theorem. Let ${(x_{n_{\alpha}})}_{\alpha}$ and ${(w_{n_{\alpha}}^*)}_{\alpha}$ be subnets such that $x_{n_{\alpha}}\to \bar x^{**}$ and $w_{n_{\alpha}}^*\to \bar w^*$, in the weak$^*$ topology of $X^{**}$ and $X^*$, respectively. Without loss of generality, we also assume that both $(\|x_{n_{\alpha}}-x\|)_{\alpha}$ and $(\|w^*_{n_{\alpha}}\|)_{\alpha}$ are convergent and, since the norm is lower semi-continuous in the weak$^*$ topology,
\begin{equation}\label{eq:normsci}
\lim_{\alpha} \|x_{n_{\alpha}}-x\|\geq \|\bar x^{**}-x\|,\quad\text{and}\quad\lim_{\alpha} \|w_{n_{\alpha}}^*\|\geq \|\bar w^*\|.
\end{equation}
We replace the previously chosen subnets in~\eqref{eq:eq3.5} and take limits, so we obtain
\begin{align*}
\inner{u^{**}-\bar x^{**}}{u^*-x^*}&\geq \dfrac{1}{2}\lim_{\alpha}\|x_{n_{\alpha}}-x\|^2+\dfrac{1}{2}\lim_{\alpha}\|w^*_{n_{\alpha}}\|^2+\inner{x-u^{**}}{\bar w^*},\\
 &\geq \dfrac{1}{2}\|\bar x^{**}-x\|^2+\dfrac{1}{2}\|\bar w^*\|^2+\inner{x-u^{**}}{\bar w^*}.
\end{align*}
Rearranging the last equation, we obtain~\eqref{eq:forfitz}. Now, combining~\eqref{eq:forfitz} and~\eqref{eq:normsci}, for all $(u^{**},u^*)\in S$,
\begin{align*}
\inner{u^{**}-\bar x^{**}}{u^*-(x^*-\bar w^*)}&\geq \dfrac{1}{2}\lim_{\alpha}\|x_{n_{\alpha}}-x\|^2+\dfrac{1}{2}\lim_{\alpha}\|w^*_{n_{\alpha}}\|^2+\inner{x-\bar x^{**}}{\bar w^*}\\
&\geq\dfrac{1}{2}\|\bar x^{**}-x\|^2+\dfrac{1}{2}\|\bar w^*\|^2-\inner{\bar x^{**}-x}{\bar w^*}\geq 0,
\end{align*}
which implies $(\bar x^{**},x^*-\bar w^*)\sim S$.
\end{proof}

So we have the following proposition.
\begin{proposition}\label{prop:main}
Let $X$ be a real Banach space, $\{T_n:X\tos X^*\}_n$ be a sequence of maximal monotone operators of type (D) and let $(\eps_n)_n$ be any sequence of positive numbers convergent to zero. If $T=\liminf T_n$, then $(x,x^*)\in T$ if, and only if, $x=\lim x_n$, where $x_n$ is a solution of
\[
x^*\in T_n(x_n)+J_{\eps_n}(x_n-x).
\]
\end{proposition}
\begin{proof}
In the proof of Lemma~\ref{lem:tech}, consider the particular case when $(x,x^*)\in T$. Choose $(u^{**},u^*)=(x,x^*)$ so we have $u_n^{**}\to u^{**}=x$ and $u_n^*\to u^*=x^*$, strongly. Thus, taking limits on~\eqref{eq:eq3.5}, since all the sequences involved are bounded, we have $x_n\to x$ and $w_n^*\to 0$, also strongly.
This proves the ``only if'' part. Conversely, if $x_n\to x$, equation~\eqref{eq:eq2} implies that $w_n^*\to 0$. Hence $x_n^*\to x^*$. Since $(x_n,x_n^*)\in T_n$, $(x,x^*)\in T$ and the proposition follows.
\end{proof}

We end this section by proving that the lower limit $T$ is a representable operator.
\begin{theorem}\label{teo:main}
Let $X$ be a real Banach space and $\{T_n:X\tos X^*\}_{n\in\N}$ be a sequence of maximal monotone type (D) operators. Then, the lower limit $T=\liminf T_n$ is representable.
\end{theorem}
\begin{proof}
We recall the definition of the Fitzpatrick function of $S^{-1}$:
\[
\varphi_{S^{-1}}(z^*,z^{**})=\inner{z^*}{z^{**}}-\inf_{(u^{**},u^*)\in S}\inner{u^*-z^*}{u^{**}-z^{**}}.
\]
Let $(x,x^*)\in X\times X^*$ be arbitrary and consider sequences $(x_n)_n$ and ${(w_n^*)}_n$ as in Lemma~\ref{lem:tech}. 
Let $\bar x^{**}$ and $\bar w^*$ be any weak$^*$-limit points of $(x_n)_n$ and ${(w_n^*)}_n$, respectively.
Also by Lemma~\ref{lem:tech}, equation~\eqref{eq:forfitz} holds for certain subnets ${(x_{n_\alpha})}_{\alpha}$ and ${(w^*_{n_\alpha})}_{\alpha}$ of ${(x_n)}_n$ and ${(w^*_n)}_n$, respectively.

Taking infimum over every $(u^*,u^{**})\in S^{-1}$ on~\eqref{eq:forfitz}, and using the definition of $\varphi_{S^{-1}}(x^*-\bar w^*,\bar x^{**})$, 
\begin{multline}\label{eq:fitzs}
\varphi_{S^{-1}}(x^*-\bar w^*,\bar x^{**})\leq \inner{\bar x^{**}-x}{\bar w^*}+\inner{x^*-\bar w^*}{\bar x^{**}}\\-\dfrac{1}{2}\lim_{\alpha}\|x_{n_{\alpha}}-x\|^2-\dfrac{1}{2}\lim_{\alpha}\|w^*_{n_{\alpha}}\|^2
\end{multline}

Now, from the definition of the Fenchel conjugate of $\varphi_{S^{-1}}$,
\begin{align}
\varphi_{S^{-1}}^*(x,x^*)&=\sup_{(a^*,a^{**})\in X^*\times X^{**}}\Big\{\inner{a^*}{x}+\inner{a^{**}}{x^{*}}-\varphi_{S^{-1}}(a^*,a^{**})\Big\}\nonumber\\
 &\geq \inner{x^*-\bar w^*}{x}+\inner{\bar x^{**}}{x^{*}}-\varphi_{S^{-1}}(x^*-\bar w^*,\bar x^{**})\label{eq:fenchelfitz}.
\end{align}
Using~\eqref{eq:fitzs} on~\eqref{eq:fenchelfitz} we obtain
\begin{multline*}
\varphi_{S^{-1}}^*(x,x^*)\geq \inner{x^*-\bar w^*}{x}+\inner{\bar x^{**}}{x^{*}}-\inner{\bar x^{**}-x}{\bar w^*}\\
-\inner{x^*-\bar w^*}{\bar x^{**}}
+\dfrac{1}{2}\lim_{\alpha}\|x_{n_{\alpha}}-x\|^2+\dfrac{1}{2}\lim_{\alpha}\|w^*_{n_{\alpha}}\|^2
\end{multline*}
Thus implying,
\begin{equation}\label{eq:new1}
\varphi_{S^{-1}}^*(x,x^*)\geq \inner{x^*}{x}
+\dfrac{1}{2}\lim_{\alpha}\|x_{n_{\alpha}}-x\|^2+\dfrac{1}{2}\lim_{\alpha}\|w^*_{n_{\alpha}}\|^2
\geq \inner{x}{x^*}.
\end{equation}

Now we prove that $T$ is representable. Define $h:X\times X^*\to\R\cup\{+\infty\}$ as 
\[
h(x,x^*)=\varphi_{S^{-1}}^*(x,x^*)=\displaystyle\sup_{(a^*,a^{**})\in X^*\times X^{**}}\Big\{\inner{a^*}{x}+\inner{a^{**}}{x^{*}}-\varphi_{S^{-1}}(a^*,a^{**})\Big\}.
\]
Then $h$ is convex and strongly lower semicontinuous and, in view of~\eqref{eq:new1}, 
\[
h(x,x^*)\geq \inner{x}{x^*},\qquad\forall\,(x,x^*)\in X\times X^*. 
\]
In addition, if $(x,x^*)\in T$ then $(x^*,x)\in S^{-1}$, so using Fact~\ref{fact:fitz}, 
\[
h(x,x^*)=\varphi_{S^{-1}}^{*}(x,x^*)=\inner{x}{x^*}.
\]
Conversely, if $h(x,x^*)=\inner{x}{x^*}$ then equation~\eqref{eq:new1} implies that $x_{n_{\alpha}}\to x$ and $w_{n_{\alpha}}^*\to 0$, strongly. In view of equation~\ref{eq:normsci} in the proof of Lemma~\ref{lem:tech}, we have $\bar x^{**}=x$ and $\bar w^*=0$, so these weak$^*$-limit points are in fact strong limit points. Since these were chosen arbitrarily, the bounded sequences $(x_n)_n$ and ${(w_n^*)}_n$ possess an unique strong limit point. Hence  $(x_n)_n$ is strongly convergent to $x$ and, by Proposition~\ref{prop:main}, $(x,x^*)\in T$.
\end{proof}

\section{The variational sum and composition}\label{sec:varsumcomp}
The \emph{variational sum} of maximal monotone operators was defined by Attouch \emph{et al.}~\cite{AtBaTh94} in Hilbert Spaces and then generalized to reflexive Banach spaces by Revalski and Th\'era~\cite{RevThe99-1,RevThe99-2}.  Recent results about this kind of sum were given by Garc\'ia \emph{et al.}~\cite{MR2583894,GarciaLassonde2012,GarLasRev06}.

Similarly, the \emph{variational composition}, which was introduced by Pennanen \emph{et al.}~\cite{PeReTh03}, also for reflexive Banach spaces. 
We now extend such notions to general Banach spaces, for maximal monotone operators of type (D). 

Given a Banach space $X$, and two maximal monotone type (D) operators $T_1,T_2:X\tos X^*$, their \emph{variational sum} is defined as follows
\[
T_1\mathop{+}_{v}T_2=\bigcap_{\mathcal{I}}\liminf_n(T_{1,\lambda_n}+T_{2,\mu_n}),
\]
where 
\begin{equation}\label{eq:calI}
\mathcal{I}=\{(\lambda_n,\mu_n)_n\subset\R^2\::\:\lambda_n,\mu_n\geq 0,\, \lambda_n+\mu_n>0,\,\lambda_n,\mu_n\to 0\}
\end{equation}
and $T_\lambda$ denotes the Moreau-Yosida regularization of $T$, for $\lambda\geq 0$ ($T_0$ simply denotes the operator $T$).  
In the same way, for a type (D) operator $T:X\tos X^*$ and $A:Y\to X$ linear and continuous, the \emph{variational composition} $(A^*TA)_v$ is defined as
\[
(A^*TA)_v=\bigcap_{\mathcal{J}}\liminf_n A^*T_{\lambda_n}A,
\]
where 
\begin{equation}\label{eq:calJ}
\mathcal{J}=\{(\lambda_n)_n\subset\R\::\:\lambda_n > 0,\, \lambda_n\to 0^+\}.
\end{equation}

Garc\'ia and Lassonde~\cite{GarciaLassonde2012}, proved that both the variational sum and composition are representable, when the underlying space was reflexive, strictly convex with strictly convex dual. 
To recover this result in the general case, we need the following result due to Marques Alves and Svaiter.

\begin{fact}[{\cite[Lemma 3.5]{BM2011}}]\label{fact:masv}
Let $T_1,T_2:X\tos X^*$ be maximal monotone operators of type (D), 
and let $h_1,h_2:X\times X^*\to \R\cup\{+\infty\}$ be representatives of $T_1$ and $T_2$, respectively.
If 
\[
\bigcup_{\lambda>0}\lambda\Big(\proy_X(\dom(h_1))-\proy_X(\dom(h_2))\Big)
\] 
is a closed subspace then $T_1+T_2$ is a maximal monotone operator of type (D). 
\end{fact}
Observe that this fact implies, in particular, that $T_1+T_2$ is of type (D) whenever $T_1$ (or $T_2$) is everywhere defined.

We have the following theorem.
\begin{theorem}\label{teo:varsum}
Let $X$ be a real Banach space and let $T_1,T_2:X\tos X^*$ be maximal monotone operators of type (D). Then their variational sum $\ds T_1\mathop{+}_v T_2$ is representable.
\end{theorem}
\begin{proof}
Take any sequence $(\lambda_n,\mu_n)_n\in\mathcal{I}$ (see~\eqref{eq:calI}). From Fact~\ref{fact:myreg}, for any $n\in\N$, the Moreau-Yosida regularizations $T_{1,\lambda_n}$ and $T_{2,\mu_n}$ of $T_1$ and $T_2$, respectively are also of type (D) and everywhere defined and, by Fact~\ref{fact:masv}, $T_{1,\lambda_n}+T_{2,\mu_n}$ is maximal monotone of type (D).  Thus, $(T_{1,\lambda_n}+T_{2,\mu_n})_n$ is a sequence of type (D) operators, so its lower limit is representable and, since arbitrary intersection of representable operators is also representable, so is the variational sum.
\end{proof}

As done in~\cite[\S 5]{GarciaLassonde2012}, we can express the variational composition in terms of a variational sum. Let $X,Y$ be real Banach spaces, $T:X\tos X^*$ be maximal monotone and $A:Y\to X$ be linear continuous. Define $T^{\#},N_A:Y\times X\tos Y^*\times X^*$, respectively as $T^{\#}(y,x)=\{0\}\times T(x)$ and
\[
N_A(y,x)=\partial\delta_{\graph(A)}=
\begin{cases}
\{(A^*x^*,-x^*)\::\:x^*\in X^*\},&\text{if }(y,x)\in A,\\
\emptyset,&\text{otherwise},
\end{cases}
\]
where $\delta_{\graph(A)}$ is the indicator function of the graph of $A$. Then
\[
y^*\in A^*TA(y)\quad\iff\quad (y^*,0)\in (T^{\#}+N_A)(y,Ay),
\]
and we have the following lemma.
\begin{lemma}\label{lem:varcomp}
Let $T:X\tos X^*$ be a maximal monotone operator of type (D) and $A:Y\to X$ be linear continuous. If $\Dom(T)=X$ then $A^*TA$ is maximal monotone of type (D).
\end{lemma}
\begin{proof}
Note that both $T^{\#}$ and $N_A$ are maximal monotone of type (D) in $Y\times X$, as $T$ is of type (D) and $N_A$ is a subdifferential. Moreover, $T^{\#}$ is everywhere defined since $T$ is. Hence, by Fact~\ref{fact:masv}, $T^{\#}+N_A$ is of type (D).

Observe that $\Dom(T^{\#}+N_A)=\Dom(N_A)=\graph(A)$, therefore
\[
T^{\#}+N_A=\left\{(y,Ay,A^*x^*,w^*-x^*)\in Y\times X\times Y^*\times X^*\::\:
\begin{array}{c}
y\in Y,\,x^*\in X^*\\
w^*\in T(A(y))
\end{array}\right\}.
\]

The maximal monotonicity of $A^*TA$ is straightforward, so remains to prove that $A^*TA$ is  maximal monotone of type (D). In view of the equivalency between the type (D) and type (NI) classes~\cite{BSMMA-TypeD}, we will prove that $A^*TA$ is of type (NI).
Given any $(v^{**},v^*)\in Y^{**}\times Y^*$, take $(u^{**},u^*)=(A^{**}v^{**},0)\in X^{**}\times X^*$, and using that $T^{\#}+N_A$ is of type (D), we have
\[
\inf_{\substack{y\in Y\\w^*\in T(A(y))\\x^*\in X^*}}\inner{v^{**}-y}{v^*-A^*x^*}+\inner{A^{**}v^{**}-Ay}{x^*-w^*}\leq 0
\]
Rearranging the last equation, 
we obtain
\[
\inf_{\substack{y\in Y\\w^*\in T(A(y))}}\inner{v^{**}-y}{v^*-A^*w^*}\leq 0,
\]
and the lemma follows.
\end{proof}

Finally, we have the following theorem.
\begin{theorem}\label{teo:varcomp}
Let $X,Y$ be real Banach spaces, let $T:X\tos X^*$ be a maximal monotone operator of type (D) and let $A:Y\to X$ be a linear continuous map. Then the variational composition $(A^*TA)_v$ is representable.
\end{theorem}
\begin{proof}
From Fact~\ref{fact:myreg} and Lemma~\ref{lem:varcomp}, for any sequence $(\lambda_n)_n\in\mathcal{J}$ (see~\eqref{eq:calJ}) and $n\in\N$, $A^*T_{\lambda_n}A$ is maximal monotone of type (D).  Thus, $(A^*T_{\lambda_n}A)_n$ is a sequence of type (D) operators, so its lower limit is representable and, so is the variational composition.
\end{proof}

\section{Acknowledgements}
The authors were partially supported by the MathAmSud regional program: project 13MATH-01.

M. Marques Alves was partially supported  Brazilian CNPq grants 305414/2011-9 and 406250/2013-8.


\begin{thebibliography}{10}

\bibitem{AtBaTh94}
H.~Attouch, J.-B. Baillon, and M.~Th{\'e}ra.
\newblock {Variational sum of monotone operators}.
\newblock {\em J. Convex Anal.}, 1(1):1--29, 1994.

\bibitem{BCP70}
H.~Brezis, M.~G. Crandall, and A.~Pazy.
\newblock {Perturbations of nonlinear maximal monotone sets in {B}anach space}.
\newblock {\em Comm. Pure Appl. Math.}, 23:123--144, 1970.

\bibitem{BS}
R.~S. Burachik and B.~F. Svaiter.
\newblock {Maximal monotone operators, convex functions and a special family of
  enlargements}.
\newblock {\em Set-Valued Anal.}, 10(4):297--316, 2002.

\bibitem{Fitz}
S.~Fitzpatrick.
\newblock {Representing monotone operators by convex functions}.
\newblock In {\em {Workshop/{M}iniconference on {F}unctional {A}nalysis and
  {O}ptimization ({C}anberra, 1988)}}, volume~20 of {\em {Proc. Centre Math.
  Anal. Austral. Nat. Univ.}}, pages 59--65. Austral. Nat. Univ., Canberra,
  1988.

\bibitem{MR2583894}
Y.~Garc{\'i}a.
\newblock {New properties of the variational sum of monotone operators}.
\newblock {\em J. Convex Anal.}, 16(3-4):767--778, 2009.

\bibitem{GarciaLassonde2012}
Y.~Garc{\'i}a and M.~Lassonde.
\newblock {Representable Monotone Operators and Limits of Sequences of Maximal
  Monotone Operators}.
\newblock {\em Set-Valued and Variational Analysis}, 20:61--73, 2012.
\newblock 10.1007/s11228-011-0178-8.

\bibitem{GarLasRev06}
Y.~Garc{\'i}a, M.~Lassonde, and J.~P. Revalski.
\newblock {Extended sums and extended compositions of monotone operators}.
\newblock {\em J. Convex Anal.}, 13(3-4):721--738, 2006.

\bibitem{JPGos0}
J.-P. Gossez.
\newblock {Op{\'e}rateurs monotones non lin{\'e}aires dans les espaces de
  {B}anach non r{\'e}flexifs}.
\newblock {\em J. Math. Anal. Appl.}, 34:371--395, 1971.

\bibitem{BSMMA-TypeD}
M.~{Marques Alves} and B.~F. Svaiter.
\newblock {On {G}ossez type ({D}) maximal monotone operators}.
\newblock {\em J. Convex Anal.}, 17(3), 2010.

\bibitem{MAS11}
M.~{Marques Alves} and B.~F. Svaiter.
\newblock {Moreau-{Y}osida regularization of maximal monotone operators of type
  ({D})}.
\newblock {\em Set-Valued Var. Anal.}, 19(1):97--106, 2011.

\bibitem{BM2011}
M.~{Marques Alves} and B.~F. Svaiter.
\newblock {On the surjectivity properties of perturbations of maximal monotone
  operators in non-reflexive {B}anach spaces}.
\newblock {\em J. Convex Anal.}, 18(1):209--226, 2011.

\bibitem{BSML}
J.~E. Mart{\'i}nez-Legaz and B.~F. Svaiter.
\newblock {Monotone operators representable by l.s.c.\ convex functions}.
\newblock {\em Set-Valued Anal.}, 13(1):21--46, 2005.

\bibitem{MLThe}
J.~E. Mart{\'i}nez-Legaz and M.~Th{\'e}ra.
\newblock {A convex representation of maximal monotone operators}.
\newblock {\em J. Nonlinear Convex Anal.}, 2(2):243--247, 2001.
\newblock Special issue for Professor Ky Fan.

\bibitem{PeReTh03}
T.~Pennanen, J.~P. Revalski, and M.~Th{\'e}ra.
\newblock {Variational composition of a monotone operator and a linear mapping
  with applications to elliptic {PDE}s with singular coefficients}.
\newblock {\em J. Funct. Anal.}, 198(1):84--105, 2003.

\bibitem{Penot04}
J.-P. Penot.
\newblock {The relevance of convex analysis for the study of monotonicity}.
\newblock {\em Nonlinear Anal.}, 58(7-8):855--871, 2004.

\bibitem{RevThe99-1}
J.~P. Revalski and M.~Th{\'e}ra.
\newblock {Generalized sums of monotone operators}.
\newblock {\em C. R. Acad. Sci. Paris S{\'e}r. I Math.}, 329(11):979--984,
  1999.

\bibitem{RevThe99-2}
J.~P. Revalski and M.~Th{\'e}ra.
\newblock {Variational and extended sums of monotone operators}.
\newblock In {\em {Ill-posed variational problems and regularization techniques
  ({T}rier, 1998)}}, volume 477 of {\em {Lecture Notes in Econom. and Math.
  Systems}}, pages 229--246. Springer, Berlin, 1999.

\bibitem{Rock70}
R.~T. Rockafellar.
\newblock {On the maximality of sums of nonlinear monotone operators}.
\newblock {\em Trans. Amer. Math. Soc.}, 149:75--88, 1970.

\bibitem{Simons96}
S.~Simons.
\newblock {The range of a monotone operator}.
\newblock {\em J. Math. Anal. Appl.}, 199(1):176--201, 1996.

\bibitem{Voisei06}
M.~D. Voisei.
\newblock {A maximality theorem for the sum of maximal monotone operators in
  non-reflexive {B}anach spaces}.
\newblock {\em Math. Sci. Res. J.}, 10(2):36--41, 2006.

\end{thebibliography}
\end{document}